\documentclass[12pt,twopage]{paper}

\usepackage{libertine}
\usepackage{amssymb}
\usepackage{amsfonts}
\usepackage{amsmath}
\usepackage{bold-extra}
\usepackage{pifont}
\usepackage{yfonts}
\usepackage{mathrsfs}
\usepackage{url}
\usepackage{hyperref}
\usepackage{fancyhdr}

\newtheorem{theorem}{\textbf{\textsc{Theorem}}}[section]
\newtheorem{definition}[theorem]{\textbf{\textsc{Definition}}}
\newtheorem{corollary}[theorem]{\textbf{\textsc{Corollary}}}
\newtheorem{exampl}[theorem]{\textbf{\textsc{Example}}}
\newtheorem{remark}[theorem]{\textbf{\textsc{Remark}}}
\newtheorem{fact}[theorem]{\textbf{\textsc{Fact}}}

\newenvironment{proof}
{\noindent\mbox{\textsf{\textbf{\textsc{Proof}}}:}}
{\hfill{\scriptsize \mbox{\underline{\texttt{\em QED}\,}$\!\big|$}}\bigskip}

\title{On \textsc{G\"odel}'s ``Much Weaker'' Assumption}

\author{{\sc Saeed \  Salehi  } \\
\textrm{\large Department of Mathematics, Statistics, and Computer Science,   University of Tabriz,}\\ \textrm{\large  Bahman 29$^{\,th}$ Boulevard,  P.O.Box~51666--16471, Tabriz, IRAN.  \, E-mail:\!~\textsf{\normalsize salehipour@tabrizu.ac.ir}} }

\begin{document}

\maketitle

\pagestyle{fancy}
\lhead[ ]{ \thepage\quad{\sc Saeed Salehi} (\textgoth{2022}) }
\chead[ ]{ }
\rhead[ ]{ {\em On G{\scriptsize \"ODEL}'s \textup{``}Much Weaker\textup{''} Assumption} }
\lfoot[ ]{ }
\cfoot[ ]{ }
\rfoot[ ]{ }

\begin{abstract}
 G\"odelian sentences of a sufficiently strong  and recursively enumerable theory, constructed in {\sc G\"odel}'s 1931 ground breaking paper on the incompleteness theorems, are unprovable if the theory is consistent; however, they could be refutable. These sentences are independent
 when the theory is so-called $\omega$-consistent; a notion introduced by {\sc G\"odel}, which is stronger than (simple) consistency,  but ``much weaker'' than soundness.
 {\sc G\"odel} goes to great lengths to show in detail that $\omega$-consistency is stronger than consistency, but never shows,  or seems to forget to say, why it is much weaker than soundness. In this paper, we study this proof-theoretic notion and compare some of its properties with those of consistency and (variants of) soundness.

\bigskip

{\sc Rosser}'s Theorem,
Consistency,
$\omega$-Consistency,
Soundness.

\end{abstract}
%


\section{Introduction}\label{sec:intro}

{\Large I}{\large N \textsc{the}} \textsc{penultimate paragraph}  of the first section  of his 1931 celebrated paper  on the incompleteness theorems, {\sc G\"odel}  wrote

\begin{quote}
\vspace{-1em}
\begin{itemize}
\item[]
The method of proof just explained can clearly be applied to any formal system that, first, [...]  and in which, second, every provable formula is true in the interpretation considered. The purpose of carrying out the above proof with full precision in what follows is, among other things, to replace the second of the assumptions just mentioned by a purely formal and much weaker one.
\hfill \cite[p.~151]{Godel}
\end{itemize}
\vspace{-1em}
\end{quote}

He began the next section with the sentence, ``We now proceed to carry out with full precision the proof sketched above''.  It is clear then that {\sc G\"odel}  sketched his proof of the first incompleteness theorem in Section~1 of \cite{Godel} for the system of {\em Principia Mathematica}, and then noted that his method of proof works for any formal system that, first, is {\em sufficiently strong} (in today's terminology) and, second, is {\em sound} (with respect to the standard model of natural numbers). He then said that in the rest of the article the proof would be carried out with full precision, while the second assumption (that of soundness) was replaced by a ``purely formal and much weaker one''. This assumption was called $\omega$-consistency by him (see Definition~\ref{def:oc} below). A question pursued in this paper is the following:

\qquad \qquad {\em Why is the purely formal notion of $\omega$-consistency much weaker than soundness?}\,\footnote{\!See also \cite[p.~141, the paragraph after the proof of Proposition 19]{Isaacson}.}

One possible answer could be the pure formality of $\omega$-consistency itself!
 {\sc G\"odel} knew that soundness (or truth) is not purely formal (what we know today from {\sc Tarski}'s Undefinability Theorem); see e.g.\ \cite{Murawski}. And, $\omega$-consistency is purely formal (arithmetically definable, see Definition~\ref{def:oic} below). Since soundness implies $\omega$-consistency, and the latter is definable while the former is not, then $\omega$-consistency should be (much) weaker than soundness. Could {\sc G\"odel} have meant in the penultimate paragraph ``to replace
the
second of the assumptions just mentioned by a purely formal and (thus) much weaker one''? In the other words, could his reason for the weakness of  $\omega$-consistency in front of soundness be the pure formality (arithmetical definability) of the former (and undefinability of the latter)?

On the other hand, the independence of the G\"odelian sentences can be guaranteed by much weaker assumptions (much weaker than $\omega$-consistncy!); $1$-consistency is more than enough. Even {\sc G\"odel} mentioned in the last page of \cite{Godel} that the consistency of the theory with its (standard) Consistency Statement is sufficient (and also necessary, see e.g.\ \cite[Theorem~35]{Isaacson}).
These stronger (than simple consistency) assumptions were all removed by Rosser \cite{Rosser36} who showed in 1936 the independence of other sentences from consistent theories (which are recursively enumerable and contain sufficient  arithmetic).
Before going to technical details, let us  quote   {\sc Smory\'nski} about $\omega$-consistency:

\begin{quote}
\vspace{-1em}
\begin{itemize}
\item[]
One weakness of {\sc G\"odel}'s original work was his introduction of the semantic notion of $\omega$-consistency. I find this notion to be pointless, but I admit many proof theorists take it seriously.  \hfill  \cite[p.~158, Remark]{Smorynski}
\end{itemize}
\vspace{-1em}
\end{quote}

It is notable that some prominent logicians, of the caliber of {\sc Henkin} \cite{Henkin},   studied, and even generalized, the concept of $\omega$-consistency. Which is a {\em syntactic} (purely formal, proof-theoretic) notion; not ``semantic''! (see Remark~\ref{remark} below). However, we can agree with {\sc Smory\'nski} that $\omega$-consistency could be ``pointless'', and may be dismissed with.

\section{$\boldsymbol\omega$-Consistency and Some of Its Semantic Properties}\label{sec:sem}
Let us fix a sufficiently strong theory $\mathbb{P}$ over an arithmetical language (which contains, $+,\times$, and probably some other constant, relation, or function symbols). This could be {\sc Peano}'s Arithmetic, which is more than enough, or some of its weaker fragments, such as ${\rm I\Sigma_1}$.
\newline\centerline{\em All our theories are \textup{(usually {\sc re})} sets of arithmetical sentences that contain  $\mathbb{P}$.} Let us be given a fixed {\sc G\"odel} coding and arithmetization by which we have the provability predicate ${\tt Pr}_T(x)$, for a fixed coding of the theory $T$, saying that ``(the sentence with code) $x$ is $T$-provable''.  We assume familiarity with the notions of $\Sigma_m$ and $\Pi_m$ formulas.

\begin{definition}[{\rm $\omega$-Consistency}]
\label{def:oc}
\noindent

\noindent
The theory $T$ is called $\omega$-consistent, when there is {\em no} formula $\varphi(x)$, with the only free variable $x$, such that $T\vdash\neg\forall x\varphi(x)$ and $T\vdash\varphi(\overline{n})$ for each $n\!\in\!\mathbb{N}$. Here, $\overline{n}$ denotes the standard term that represents $n$.
\hfill\ding{71}\end{definition}

\begin{exampl}[{\rm of an $\omega$-consistent and an $\omega$-inconsistent theory}]\label{example}
\noindent

\noindent
Every sound theory is $\omega$-consistent. To see a natural $\omega$-{\em in}consistent theory, let us consider the negation of the (formal) Induction Principle. For a formula $\varphi(x)$, the formal Induction Principle of $\varphi$ is
\newline\centerline{${\rm IND}_\varphi:\quad \varphi(\overline{0})\wedge\forall x\,[\varphi(x)\!\rightarrow\!\varphi(x\!+\!\overline{1})]
\longrightarrow\forall x\,\varphi(x).$}
It is known that  the IND of formulas with smaller complexity do not imply the IND of formulas with higher complexity. So, $\neg{\rm IND}_\varphi$ could be consistent with some weak arithmetical theories, when $\varphi$ is a sufficiently complex formula. We show that $\neg{\rm IND}_\varphi$ entails an $\omega$-inconsistency.
First, note that $\neg{\rm IND}_\varphi\vdash\neg\forall x\varphi(x)$.
Second, we have
$\neg{\rm IND}_\varphi\vdash
\varphi(\overline{n})\!\rightarrow\!
\varphi(\overline{n\!+\!1})$ for every $n\!\in\!\mathbb{N}$, and so by (meta-)induction on $n$ one can show  $\neg{\rm IND}_\varphi\vdash
\varphi(\overline{n})$ for every $n\!\in\!\mathbb{N}$, noting that
 $\neg{\rm IND}_\varphi\vdash\varphi(\overline{0})$.  Therefore, $\neg{\rm IND}_\varphi$ is $\omega$-inconsistent.
\hfill\ding{71}\end{exampl}

We will use the following result of {\sc Isaacson}  \cite[Theorem 21]{Isaacson}, which is the $\omega$-version of {\sc Lindenbaum}'s Lemma:

\begin{fact}[$\omega\textrm{-}{\tt Con}_{T}\Longrightarrow\forall\psi\!\!: \omega\textrm{-}{\tt Con}_{T+\psi}\vee\omega\textrm{-}{\tt Con}_{T+\neg\psi}$]\label{fact:1}
\noindent

\noindent
If $T$ is $\omega$-consistent, then for every sentence $\psi$ either $T\!+\!\psi$ or $T\!+\!\neg\psi$ is $\omega$-consistent.
\hfill \ding{113}
\end{fact}

For a proof, assume (for the sake of a contradiction) that both $T\!+\!\psi$ and $T\!+\!\neg\psi$ are $\omega$-inconsistent. Then for some formulas $\alpha(x)$ and $\beta(x)$ we have
 $T\!+\!\psi\vdash\neg\forall x\,\alpha(x)$ and $T\!+\!\psi\vdash\alpha(\overline{n})$ for each $n\!\in\!\mathbb{N}$, also  $T\!+\!\neg\psi\vdash\neg\forall x\,\beta(x)$ and $T\!+\!\neg\psi\vdash\beta(\overline{n})$ for each $n\!\in\!\mathbb{N}$.
By Deduction Theorem we have $T\vdash\neg\forall x\,[\psi\!\rightarrow\!\alpha(x)]$ and $T\vdash\neg\forall x\,[\neg\psi\!\rightarrow\!\beta(x)]$, and so by Classical Logic we have
(I) $T\vdash\neg\forall x\,\big([\psi\!\rightarrow\!\alpha(x)]
\wedge[\neg\psi\!\rightarrow\!\beta(x)]\big)$.
Again, by Deduction Theorem, for every $n\!\in\!\mathbb{N}$ we have (II) $T\vdash[\psi\!\rightarrow\!\alpha(\overline{n})]
\wedge[\neg\psi\!\rightarrow\!\beta(\overline{n})]$. Now, (I) and (II) imply that $T$ is not $\omega$-consistent, which contradicts the assumption.
\hfill  {\scriptsize\texttt{\em QED}}

An interesting consequence of Fact~\ref{fact:1}, in the light of Example~\ref{example}, is that if $T$ is $\omega$-consistent then $T\!+\!{\rm IND}_\varphi$, for any formula $\varphi(x)$, is $\omega$-consistent too. So is the theory $T\!+\!\{{\rm IND}_{\varphi_1},\cdots,{\rm IND}_{\varphi_n}\}$, for any finite set of formulas $\{\varphi_1(x),\cdots,\varphi_n(x)\}$. Thus, any $\omega$-consistent theory is consistent with {\sc Peano}'s Arithmetic.

Our next observation is that $\omega$-consistency is preserved by adding true $\Sigma_3$-sentences; this was first proved for (adding true) $\Pi_1$-sentences in \cite[Theorem~22]{Isaacson} with a proof attributed
to {\sc Kreisel} 2005. Let us recall that our base theory $\mathbb{P}$ is $\Sigma_1$-complete, i.e., can prove every true $\Sigma_1$-sentence.

\begin{theorem}[$\omega\textrm{-}{\tt Con}_T\wedge\sigma\!\in\!\Sigma_3\textrm{-}{\rm Th}(\mathbb{N})\Longrightarrow\omega\textrm{-}{\tt Con}_{T+\sigma}\wedge\neg\omega\textrm{-}{\tt Con}_{T+\neg\sigma}$]\label{thm:sigma3}
\noindent

\noindent
If $T$ is an $\omega$-consistent theory and $\sigma$ is a true $\Sigma_3$-sentence, then the theory $T\!+\!\sigma$ is $\omega$-consistent and the theory $T\!+\!\neg\sigma$ is $\omega$-inconsistent.
\end{theorem}
\begin{proof}

\noindent
Write $\sigma=\exists x\,\pi(x)$ for a $\Pi_2$-formula $\pi$. Since $\sigma$ is true, then there exists some $k\!\in\!\mathbb{N}$ such that $\mathbb{N}\vDash\pi(\overline{k})$. Write $\pi(\overline{k})=\forall y\,\theta(y)$ for some $\Sigma_1$-formula $\theta$. Then for every $n\!\in\!\mathbb{N}$ we have $\mathbb{N}\vDash\theta(\overline{n})$. So, by the $\Sigma_1$-completeness of $\mathbb{P}$ we have  $\mathbb{P}\vdash\theta(\overline{n})$ for each $n\!\in\!\mathbb{N}$. Now, from $\neg\sigma\vdash\neg\pi(\overline{k})$ we have $\neg\sigma\vdash\neg\forall x\,\theta(x)$. Thus, $\mathbb{P}\!+\!\neg\sigma$ is $\omega$-inconsistent; so is $T\!+\!\neg\sigma$. Since $T$ is $\omega$-consistent, then by Fact~\ref{fact:1}, $T\!+\!\sigma$ must be $\omega$-consistent.
\end{proof}

Later, we will see that this result is optimal: adding a true $\Pi_3$-sentence to an $\omega$-consistent theory does not necessarily result in an $\omega$-consistent theory (see Corollary~\ref{cor:api3} below). Let us now note that $\omega$-consistency implies $\Pi_3$-soundness.

\begin{corollary}[$\omega\textrm{-}{\tt Con}_T\Longrightarrow\Pi_3\textrm{-}{\tt Sound}_T$]\label{cor:pi3}
\noindent

\noindent
Every $\omega$-consistent theory is $\Pi_3$-sound, i.e., every provable $\Pi_3$-sentence of it is true.
\end{corollary}
\begin{proof}

\noindent
If $T$ is $\omega$-consistent, and $\pi$ is a $\Pi_3$-sentence such that $T\vdash\pi$, then $\pi$ must be true, since otherwise $\neg\pi$ would be a true $\Sigma_3$-sentence, and so by Theorem~\ref{thm:sigma3} the theory $T\!+\!\neg\pi$ would be $\omega$-consistent, but this is a contradiction since $T\!+\!\neg\pi$ is inconsistent by the assumption $T\vdash\pi$.
\end{proof}

We now note that the notion of $\omega$-consistency is arithmetically definable.

\begin{definition}[$\mho_T(\varphi)$: {\rm the formula $\varphi$ is a witness for the $\omega$-inconsistency of $T$}]
\label{def:oic}
\noindent

\noindent
 Let $\mho_T(\varphi)$ be the following formula: ${\tt Pr}_T(\ulcorner\neg\forall v\varphi(v)\urcorner)\wedge\forall w\,{\tt Pr}_T(\ulcorner\varphi(\overline{w})\urcorner)$.
  Here $\ulcorner\alpha\urcorner$ (which is a term in the language of $\mathbb{P}$) denotes the {\sc G\"odel} code of the expression $\alpha$.
   Let $\omega\textrm{-}{\tt Con}_T$ be the sentence  $\neg\exists\chi\,\mho_T(\chi)$.
\hfill\ding{71}\end{definition}

We note that when $T$ is an {\sc re} theory, then ${\tt Pr}_T(x)$ is a $\Sigma_1$-formula, so $\mho_T(x)$ is a $\Pi_2$-formula, thus $\omega\textrm{-}{\tt Con}_T$ is a $\Pi_3$-sentence.

As far as we know, the first  proof of the weakness of $\omega$-consistency with respect to  soundness appeared in print at \cite{Kreisel}, what is referred to as ({\sc Kreisel} 1955) in \cite[Proposition~19]{Isaacson}.

\begin{definition}[{\rm Kreiselian $\Sigma_3$-Sentences of $T$, $\kappa$: I am $\omega$-inconsistent with $T$}]
\label{def:kappa}
\noindent

\noindent
For a theory $T$, any $\Sigma_3$-sentence $\kappa$ that satisfies $\mathbb{P}\vdash\kappa\!\leftrightarrow\!\neg
\omega\textrm{-}{\tt Con}_{T+\kappa}$ is called a Kreiselian sentence of $T$.
\hfill\ding{71}\end{definition}

If we think for a moment that $\omega$-consistency equals to soundness, then Kreiselian sentences correspond to the Liar sentences. Now, {\sc Kreisel}'s proof of the non-equality of $\omega$-consistency with soundness corresponds to the classical proof of {\sc Tarski}'s Undefinability Theorem. Let us note that by Diagonal Lemma there exist some Kreiselian sentences for any {\sc re} theory, be it $\omega$-consistent or not.

\begin{theorem}[$\omega\textrm{-}{\tt Con}_T\Longrightarrow\omega\textrm{-}{\tt Con}_{T+\kappa}\;\&\;
\mathbb{N}\nvDash\kappa$]\label{thm:kappa}
\noindent

\noindent
If $T$ is {\sc re} and $\omega$-consistent, and $\kappa$ is a Kreiselian sentence of $T$, then $\kappa$ is false and $T\!+\!\kappa$ is $\omega$-consistent.
\end{theorem}
\begin{proof}

\noindent
If $\kappa$ were true,  then by Definition~\ref{def:kappa} and the soundness of $\mathbb{P}$, the theory $T\!+\!\kappa$ would be $\omega$-inconsistent. But by Theorem~\ref{thm:sigma3}, and the assumed truth of the $\Sigma_3$-sentence $\kappa$, the theory $T\!+\!\kappa$ should be $\omega$-consistent; a contradiction. Thus, $\kappa$ is false; and so, by the soundness of $\mathbb{P}$, the theory $T\!+\!\kappa$ is $\omega$-consistent.
\end{proof}

So, for an {\sc re} $\omega$-consistent theory $T$ and a Kreiselian sentence $\kappa$ of $T$, the theory $T\!+\!\kappa$ is $\omega$-consistent but not sound (since $\kappa$ is false); $T\!+\!\kappa$ is not even $\Sigma_3$-sound.

\begin{corollary}[$\omega\textrm{-}{\tt Con}_T\,\not\!\!\Longrightarrow\!\Sigma_3\textrm{-}{\tt Sound}_T,\; \Sigma_m\textrm{-}{\tt Sound}_T\,\not\!\!\Longrightarrow\!\omega\textrm{-}{\tt Con}_T$]\label{cor:s3n}
\noindent

\noindent
$\omega$-consistency does not imply $\Sigma_3$-soundness, and for any $m\!\in\!\mathbb{N}$, $\Sigma_m$-soundness does not imply $\omega$-consistency.
\end{corollary}

\begin{proof}

\noindent
It is rather easy to see that $\Sigma_m$-soundness (the truth of provable $\Sigma_m$-sentences) is equivalent to consistency with $\Pi_m\textrm{-}{\rm Th}(\mathbb{N})$,  the set of true $\Pi_m$-sentences. By \cite[Theorem~2.5]{SalSer} there exists a true $\Pi_{m+1}$-sentence $\gamma$ such that $\mathbb{P}\!+\!\Pi_m\textrm{-}{\rm Th}(\mathbb{N})\nvdash\gamma$. So, the theory $U=\mathbb{P}\!+\!\Pi_m\textrm{-}{\rm Th}(\mathbb{N})\!+\!\neg\gamma$ is consistent. We show that $U$ is not $\omega$-consistent. Write $\gamma=\forall x\,\sigma(x)$ for some $\Sigma_m$-formula $\sigma$. By the truth of $\gamma$  we have $\Pi_m\textrm{-}{\rm Th}(\mathbb{N})\vdash\sigma(\overline{n})$ for each $n\!\in\!\mathbb{N}$. So, we have $U\vdash\neg\forall x\,\sigma(x)$ and $U\vdash\sigma(\overline{n})$ for each $n\!\in\!\mathbb{N}$. Thus, $U$ is not $\omega$-consistent, but it is $\Sigma_m$-sound (being consistent with $\Pi_m\textrm{-}{\rm Th}(\mathbb{N})$). However, $U$ is not {\sc re}; let us consider its sub-theory  $T=\mathbb{P}\!+\!\neg\forall x\,\sigma(x)\!+\!
\{\sigma(\overline{n})\}_{n\in\mathbb{N}}$. The theory $T$ is {\sc re} and $\Sigma_m$-sound, but not $\omega$-consistent.
\end{proof}

It is worth noting that while soundness implies $\omega$-consistency, $\Sigma_m$-soundness, even for large $m$'s, does not imply $\omega$-consistency.
We can now show the optimality of Theorem~\ref{thm:sigma3}.

\begin{corollary}[$\omega\textrm{-}{\tt Con}_T\wedge\pi\!\in\!\Pi_3\textrm{-}{\rm Th}(\mathbb{N})\,\not\!\!\Longrightarrow\!\omega\textrm{-}{\tt Con}_{T+\pi}$]\label{cor:api3}
\noindent

\noindent
Adding a true $\Pi_3$-sentence to an $\omega$-consistent theory does not necessarily result in an $\omega$-consistent theory.
\end{corollary}
\begin{proof}

\noindent
Let $\kappa_0$ be a Kreiselian sentence of $\mathbb{P}$. Then, by Theorem~\ref{thm:kappa}, the theory  $T_0=\mathbb{P}\!+\!\kappa_0$ is $\omega$-consistent and $\neg\kappa_0$ is a true $\Pi_3$-sentence. But $T_0\!+\!\neg\kappa_0$ is not even consistent.
\end{proof}

Finally, we can show that adding a Kreiselian sentence or its negation to a sound theory results, in both cases, in an $\omega$-consistent theory (cf.\ Theorem~\ref{thm:rosser} below).

\begin{corollary}[$\mathbb{N}\vDash T\Longrightarrow\omega\textrm{-}{\tt Con}_{T+\kappa}\wedge\omega\textrm{-}{\tt Con}_{T+\neg\kappa}$]\label{cor:k}
\noindent

\noindent
If $T$ is a sound {\sc re} theory and $\kappa$ is a Kreiselian sentence of $T$, then both $T\!+\!\kappa$ and $T\!+\!\neg\kappa$ are $\omega$-consistent.
\end{corollary}
\begin{proof}

\noindent
The theory $T\!+\!\neg\kappa$ is sound and the theory  $T\!+\!\kappa$ is $\omega$-consistent by Theorem~\ref{thm:kappa}.
\end{proof}

\section{Some Syntactic Properties of $\boldsymbol\omega$-Consistency}\label{sec:syn}

Let us begin this section, like the previous one, with another interesting result of {\sc Isaacson}  \cite[Theorem 20]{Isaacson}; see \cite[Proposition 3.2]{SalSer} for a generalization.

\begin{fact}[$\omega\textrm{-}{\tt Con}_{T}\wedge{\tt Complete}_T\Longrightarrow T\!=\!{\rm Th}(\mathbb{N})$]\label{fact:2}
\noindent

\noindent
True Arithmetic, ${\rm Th}(\mathbb{N})$, is the only  $\omega$-consistent  theory which is complete.
\hfill \ding{113}
\end{fact}


For a proof, note that a complete and $\omega$-consistent theory $T$ is $\Pi_3$-sound by Corollary~\ref{cor:pi3}. So, we have    $(\mathcal{C}_2)$
$T\vdash\Sigma_2\textrm{-}{\rm Th}(\mathbb{N})\!\cup\!\Pi_2\textrm{-}{\rm Th}(\mathbb{N})$; since if $\eta\!\in\!\Sigma_2\textrm{-}{\rm Th}(\mathbb{N})\!\cup\!\Pi_2\textrm{-}{\rm Th}(\mathbb{N})$ and $T\nvdash\eta$, then $T\vdash\neg\eta$, by the completeness of $T$, which would contradict the $\Pi_3$-soundness of $T$.
We now show, by induction on $m$,  that $(\mathcal{C}_m)$ $T\vdash\Sigma_m\textrm{-}{\rm Th}(\mathbb{N})\!\cup\!\Pi_m\textrm{-}{\rm Th}(\mathbb{N})$. For proving $(\mathcal{C}_m\!\Rightarrow\!\mathcal{C}_{m+1})$ suppose  that $(\mathcal{C}_m)$ holds.
It is easy to see that $\Pi_m\textrm{-}{\rm Th}(\mathbb{N})\vdash\Sigma_{m+1}\textrm{-}{\rm Th}(\mathbb{N})$; so by $(\mathcal{C}_m)$ we already have $T\vdash\Sigma_{m+1}\textrm{-}{\rm Th}(\mathbb{N})$. We now show $T\vdash\Pi_{m+1}\textrm{-}{\rm Th}(\mathbb{N})$. Let $\pi$ be a true $\Pi_{m+1}$-sentence; write $\pi\!=\!\forall x\,\sigma(x)$ for some  $\Sigma_m$-formula $\sigma$. For every $n\!\in\!\mathbb{N}$ we have $\mathbb{N}\vDash\sigma(\overline{n})$, so $\Sigma_m\textrm{-}{\rm Th}(\mathbb{N})\vdash\sigma(\overline{n})$, thus by  $(\mathcal{C}_m)$ we have $T\vdash\sigma(\overline{n})$. Now,   the $\omega$-consistency of $T$ implies $T\nvdash\neg\forall x\,\sigma(x)$; so from the completeness of $T$ we have $T\vdash\forall x\,\sigma(x)$, thus  $T\vdash\pi$.
\hfill {\scriptsize\texttt{\em QED}}

\begin{corollary}[$\lim_\subseteq\omega\textrm{-}{\tt Con}\neq\omega\textrm{-}{\tt Con}$]\label{cor:lindenb}
\noindent

\noindent
The limit (union) of a chain of $\omega$-consistent theories is not necessarily $\omega$-consistent.
\end{corollary}
\begin{proof}

\noindent
Let $\kappa_0$ be a Kreiselian sentence of $\mathbb{P}$ and put $T_0=\mathbb{P}\!+\!\kappa_0$. Then $T_0$ is $\omega$-consistent by Theorem~\ref{thm:kappa}. Now, by Fact~\ref{fact:1} one can expand $T_0$ in stages $T_0\!\subseteq\!T_1\!\subseteq\!T_2\!\subseteq\!\cdots$ in a way that each $T_m$ is $\omega$-consistent and their union $T^\ast=\bigcup_mT_m$ is complete. But by Fact~\ref{fact:2}, $T^\ast$ cannot be $\omega$-consistent  since $\mathbb{N}\nvDash\kappa_0$ by Theorem~\ref{thm:kappa} and so $T^\ast\neq{\rm Th}(\mathbb{N})$. Thus, the limit $T^\ast$ of the chain $\{T_m\}_m$ of $\omega$-consistent theories  is not   $\omega$-consistent.
\end{proof}

\begin{remark}[{\rm Non-Semanticity of $\omega$-Consistency}]\label{remark}
\noindent

\noindent
Fact~\ref{fact:2} enables us to show that $\omega$-consistency is not a semantic (model-theoretic) notion.
Assume, for the sake of a contradiction,  that for a class $\mathscr{C}$ of structures (over the language of $\mathbb{P}$) and for every theory $T$,

\qquad \qquad (\ding{75}) \quad $T$ is $\omega$-consistent if and only if $\mathcal{M}\!\vDash\!T$ for some $\mathcal{M}\!\in\!\mathscr{C}$.

A candidate for such a $\mathscr{C}$ that comes to mind naturally is the class of $\omega$-type structures: a model $\mathfrak{A}$ is called $\omega$-type, when there is no formula $\varphi(x)$ such that $\mathfrak{A}\!\vDash\!\neg\forall x\,\varphi(x)$ and at the same time $\mathfrak{A}\!\vDash\!\varphi(\overline{n})$ for every $n\!\in\!\mathbb{N}$. It is clear that if a theory has an $\omega$-type model, then it is an $\omega$-consistent theory.

For showing the impossibility of (\ding{75}), take $T_0$ to be an unsound $\omega$-consistent theory (as in e.g.\ the proof of Corollary~\ref{cor:lindenb}) and assume that for $\mathcal{M}_0\!\in\!\mathscr{C}$ we have  $\mathcal{M}_0\!\vDash\!T_0$. Then the full first-order theory ${\rm Th}(\mathcal{M}_0)$ of $\mathcal{M}_0$, the set of all sentences that are true in $\mathcal{M}_0$, is a complete $\omega$-consistent theory. So, by Fact~\ref{fact:2}, ${\rm Th}(\mathcal{M}_0)$ should be equal to ${\rm Th}(\mathbb{N})$; thus $\mathcal{M}_0\equiv\mathbb{N}$ whence  $\mathbb{N}\!\vDash\!T_0$, a contradiction.
\hfill\ding{71}\end{remark}

We now show that Theorem~\ref{thm:sigma3} can be formalized in $\mathbb{P}$.

\begin{theorem}[$\sigma\!\in\!\Sigma_3\Longrightarrow
\mathbb{P}\vdash\sigma\!\wedge\!\omega\textrm{-}{\tt Con}_T\!\rightarrow\!\omega\textrm{-}{\tt Con}_{T+\sigma}$]\label{thm:formal}
\noindent

\noindent
For every $\Sigma_3$-sentence $\sigma$ and any theory $T$  we have $\mathbb{P}\vdash\sigma\!\wedge\!\omega\textrm{-}{\tt Con}_T\!\rightarrow\!\omega\textrm{-}{\tt Con}_{T+\sigma}$.
\end{theorem}
\begin{proof}

\noindent
The proofs of Theorem~\ref{thm:sigma3} and Fact~\ref{fact:1} can be formalized in $\mathbb{P}$ with some hard work. We now present a more direct proof for Theorem~\ref{thm:sigma3} whose formalizability in $\mathbb{P}$ is straightforward. Suppose that $\omega\textrm{-}{\tt Con}_T$ and that $\sigma$ is a true $\Sigma_3$-sentence. If $\neg\omega\textrm{-}{\tt Con}_{T+\sigma}$ then for some formula $\varphi(x)$ we have $T\!+\!\sigma\vdash\neg\forall x\,\varphi(x)$ and $T\!+\!\sigma\vdash\varphi(\overline{n})$ for every $n\!\in\!\mathbb{N}$. Write $\sigma=\exists x\,\pi(x)$ for a $\Pi_2$-formula $\pi$. Since $\sigma$ is true, then there exists some $u$($\in\!\mathbb{N}$) such that $\pi(u)$ is true. Write $\pi(u)=\forall y\,\theta(y)$ for some $\Sigma_1$-formula $\theta$. Then $\theta(z)$ is true for every $z$. So, by the $\Sigma_1$-completeness of $T$ we have (\ding{92}) $T\vdash\theta(\overline{n})$ for each $n\!\in\!\mathbb{N}$.
For reaching to a contradiction, we show that $T$ is $\omega$-inconsistent and the formula $\theta(x)\wedge[\pi(u)\!\rightarrow\!\varphi(x)]$ is a witness for that.
By Deduction Theorem we have  $T\vdash\sigma\!\rightarrow\!\neg\forall x\,\varphi(x)$ and so  $T\vdash\pi(u)\!\rightarrow\!\neg\forall x\,[\pi(u)\!\rightarrow\!\varphi(x)]$ therefore  $T\vdash\neg\forall y\,\theta(y)\vee\neg\forall x\,[\pi(u)\!\rightarrow\!\varphi(x)]$, thus (i) $T\vdash\neg\forall x\,\big(\theta(x)\wedge[\pi(u)\!\rightarrow\!\varphi(x)]\big)$.
On the other hand, for every $n\!\in\!\mathbb{N}$ we have $T\vdash\pi(u)\!\rightarrow\!\varphi(\overline{n})$, which by (\ding{92}) implies that (ii) $T\vdash\theta(\overline{n})\wedge[\pi(u)
\!\rightarrow\!\varphi(\overline{n})]$ for each $n\!\in\!\mathbb{N}$. Thus, by (i) and (ii) the theory $T$ is $\omega$-inconsistent, a contradiction. So, $T\!+\!\sigma$ must be $\omega$-consistent.
\end{proof}

As a corollary, we show that all the Kreiselian sentences  are $\mathbb{P}$-provably equivalent to one another.

\begin{corollary}[$\mathbb{P}\vdash\kappa\!\equiv\!\neg
\omega\textrm{-}{\tt Con}_{T}$]\label{cor:kk}
\noindent

\noindent
If  $\kappa$ is a Kreiselian sentence of the {\sc re} theory $T$, then $\mathbb{P}\vdash\kappa\!\leftrightarrow\!\neg
\omega\textrm{-}{\tt Con}_{T}$.
\end{corollary}
\begin{proof}

\noindent
Argue inside $\mathbb{P}$: If
$\kappa$ then
$\neg\omega\textrm{-}{\tt Con}_{T+\kappa}$ by Definition~\ref{def:kappa}, which implies  $\neg\omega\textrm{-}{\tt Con}_{T}$
by Theorem~\ref{thm:formal} (noting that $\kappa\!\in\!\Sigma_3$);
 therefore, $\kappa$ implies $\neg\omega\textrm{-}{\tt Con}_{T}$.
Conversely, if
$\neg
\omega\textrm{-}{\tt Con}_{T}$ then $\neg
\omega\textrm{-}{\tt Con}_{T+\kappa}$ and so $\kappa$ by Definition~\ref{def:kappa}; therefore, $\neg
\omega\textrm{-}{\tt Con}_{T}$ implies $\kappa$.
\end{proof}

By the $\mathbb{P}$-provable equivalence of $\kappa$ with $\neg\omega\textrm{-}{\tt Con}_{T}$, we have the following corollary which is the $\omega$-version of  {\sc G\"odel}'s Second Incompleteness Theorem, that was first proved by {\sc Rosser} \cite{Rosser37} (see also \cite[p.~{\sf xxxi}]{Boolos}).

\begin{corollary}[$\omega\textrm{-}{\tt Con}_T\Longrightarrow\omega\textrm{-}{\tt Con}_{T+\neg\omega\textrm{-}{\tt Con}_T}$]\label{cor:g2}
\noindent

\noindent
If the {\sc re} theory $T$ is $\omega$-consistent, then so is $T\!+\!\neg\omega\textrm{-}{\tt Con}_T$.
\hfill \ding{113}
\end{corollary}

Thus far, we have seen some $\omega$-versions of {\sc Lindenbaum}'s lemma and also {\sc G\"odel}'s first and   second incompleteness theorems. We do not claim   novelty for any of these results;\footnote{\!See e.g.\ the \texttt{\scriptsize \url{https://t.ly/GmquO}} link of MathOverFlow whose Proposition 1 (due to {\sc Je\v{r}\'{a}bek}) is half of our Theorem~\ref{thm:sigma3}.} nevertheless, the following $\omega$-version of {\sc Rosser}'s incompleteness theorem seems to be new.

\begin{theorem}[$\omega\textrm{-}{\tt Con}_T\Longrightarrow\exists\rho\!\in\!\Pi_3\textrm{-}
{\rm Th}(\mathbb{N})\!\!:\omega\textrm{-}{\tt Con}_{T+\rho}\wedge\omega\textrm{-}{\tt Con}_{T+\neg\rho}$]\label{thm:rosser}
\noindent

\noindent
If $T$ is an $\omega$-consistent {\sc re} theory, then there exists some true $\Pi_3$-sentence $\rho$ such that both $T\!+\!\rho$ and $T\!+\!\neg\rho$ are $\omega$-consistent.
\end{theorem}
\begin{proof}

\noindent
By  Diagonal Lemma there exists a $\Pi_3$-sentence $\rho$ such that (see Definition~\ref{def:oic})

\qquad \qquad \qquad (\ding{99}) \quad {$\mathbb{P}\vdash\rho\leftrightarrow
\forall\chi\,[\mho_{T+\neg\rho}(\chi)\!\rightarrow\!
\exists\xi\!<\!\chi\,\mho_{T+\rho}(\xi)]$.}

\begin{enumerate}
  \item[{\large \textgoth{a}}.] We first show that $T\!+\!\rho$ is $\omega$-consistent.

      Assume not; then for some (fixed, standard) formula $\boldsymbol\varphi(x)$ the $\Pi_2$-sentence $\mho_{T+\rho}(\boldsymbol\varphi)$ is true, so we have  $\Pi_2\textrm{-}{\rm Th}(\mathbb{N})\vdash\mho_{T+\rho}(\boldsymbol\varphi)$. Also, by Fact~\ref{fact:1},  $T\!+\!\neg\rho$ is $\omega$-consistent. Thus, $U=T\!+\!\neg\rho\!+\!\Pi_2\textrm{-}{\rm Th}(\mathbb{N})$ is consistent by Corollary~\ref{cor:pi3}. Now, by (\ding{99}) we have
{$U\vdash
\exists\chi[\mho_{T+\neg\rho}(\chi)\!\wedge\!
\forall\xi\!<\!\chi\neg\mho_{T+\rho}(\xi)]$.}
Since   $U\vdash\mho_{T+\rho}(\boldsymbol\varphi)$ then we have    $U\vdash
\exists\chi\!\leqslant\!\boldsymbol\varphi\,\mho_{T+\neg\rho}(\chi)$.
But by the $\omega$-consistency of the theory $T+\neg\rho$, the $\Sigma_2$-sentence $\forall\chi\!\leqslant\!\boldsymbol\varphi
\neg\mho_{T+\neg\rho}(\chi)$ is true, and so should be $\Pi_2\textrm{-}{\rm Th}(\mathbb{N})$-provable. Thus,   $U$ is inconsistent; a contradiction. Therefore, the theory  $T\!+\!\rho$ must be $\omega$-consistent.

  \item[{\large \textgoth{b}}.] We now show that $T\!+\!\neg\rho$ is $\omega$-consistent.

If not, then by Fact~\ref{fact:1},  $T\!+\!\rho$ should be $\omega$-consistent, and so $U=T\!+\!\rho\!+\!\Pi_2\textrm{-}{\rm Th}(\mathbb{N})$ should be consistent by Corollary~\ref{cor:pi3}. Also, for some   formula $\boldsymbol\varphi(x)$ we should have $\Pi_2\textrm{-}{\rm Th}(\mathbb{N})\vdash\mho_{T+\neg\rho}(\boldsymbol\varphi)$.
Now, by (\ding{99}) we have $U\vdash\exists\xi\!<\!\boldsymbol\varphi\,\mho_{T+\rho}(\xi)$.
But $\forall\xi\!<\!\boldsymbol\varphi\,\neg\mho_{T+\rho}(\xi)$ is a true $\Sigma_2$-sentence by the $\omega$-consistency of the theory $T\!+\!\rho$. So,   $\forall\xi\!<\!\boldsymbol\varphi\,\neg\mho_{T+\rho}(\xi)$ is $\Pi_2\textrm{-}{\rm Th}(\mathbb{N})$-provable, which implies that $U$ is inconsistent; a contradiction. Therefore, the theory $T\!+\!\neg\rho$ must be $\omega$-consistent too.
\end{enumerate}
So, both  $T\!+\!\rho$ and $T\!+\!\neg\rho$ are $\omega$-consistent, whence the $\Pi_3$-sentence $\rho$ is true (by the soundness of $\mathbb{P}$).
\end{proof}

Note that Theorem~\ref{thm:rosser} is optimal in a sense, since by Theorem~\ref{thm:sigma3} for no true $\Sigma_3$-sentence $\sigma$ can the theory $T\!+\!\neg\sigma$ be $\omega$-consistent.
We end the paper with the observation that Theorem~\ref{thm:rosser} can be formalized in $\mathbb{P}$.

\begin{corollary}[$\omega\textrm{-}{\tt Con}_T\rightarrow\rho\not\rightarrow\omega\textrm{-}{\tt Con}_T$]\label{cor:formalr}
\noindent

\noindent
If  $\rho$ is a $\Pi_3$-sentence that was constructed for the {\sc re}  theory $T$ in the Proof of Theorem~\ref{thm:rosser} (\ding{99}), then

(1) $\mathbb{P}\vdash\omega\textrm{-}{\tt Con}_T \longrightarrow \rho \wedge \omega\textrm{-}{\tt Con}_{T+\rho} \wedge \omega\textrm{-}{\tt Con}_{T+\neg\rho}$, and

(2) if $T$ is $\omega$-consistent, then $T\nvdash\rho\rightarrow\omega\textrm{-}{\tt Con}_T$; moreover, $T\!+\!\rho\!+\!\neg\omega\textrm{-}{\tt Con}_T$ is $\omega$-consistent.
\end{corollary}
\begin{proof}

\noindent
 For part  (2) it suffices to note that since $T\!+\!\rho\!+\!\neg\omega\textrm{-}{\tt Con}_{T+\rho}$ is $\omega$-consistent by Theorem~\ref{thm:rosser} and Corollary~\ref{cor:g2}, and  $\mathbb{P}\vdash\neg\omega\textrm{-}{\tt Con}_{T+\rho}\leftrightarrow\neg\omega\textrm{-}{\tt Con}_T$ by part (1), then  $T\!+\!\rho\!+\!\neg\omega\textrm{-}{\tt Con}_T$
 is $\omega$-consistent as well.
\end{proof}


\end{document}